\documentclass[a4paper,leqno,11pt]{amsart}
\usepackage{amsmath}
\usepackage{hyperref}
\usepackage{amsfonts}
\usepackage{amsmath,amscd}
\usepackage{amssymb}
\usepackage{amsthm}
\usepackage{mathtools}
\usepackage{mathrsfs}
\usepackage{graphicx}
\usepackage{xypic}
\newtheorem{definition}{Definition}[section]
\newtheorem{theorem}[definition]{Theorem}
\newtheorem{remark}[definition]{Remark}
\newtheorem{final Remarks}[definition]{Final Remarks}
\newtheorem{lemma}[definition]{Lemma}

\newtheorem{corollary}[definition]{Corollary}

\numberwithin{equation}{section}
\newcommand{\colim}{\mathit{colim}}
\begin{document}
\title{Homotopy Inertia Groups and Tangential Structures}
\vspace{2cm}

\author{Ramesh Kasilingam}

\email{rameshkasilingam.iitb@gmail.com  ; mathsramesh1984@gmail.com  }
\address{Statistics and Mathematics Unit,
Indian Statistical Institute,
Bangalore Centre, Bangalore - 560059, Karnataka, India.}

\date{}

\maketitle
\begin{abstract}
We show that if $M$ and $N$ have the same homotopy type of simply connected closed smooth $m$-manifolds such that the integral and mod-$2$ cohomologies of $M$ vanish in odd degrees, then their homotopy inertia groups are equal. Let $M^{2n}$ be a closed $(n-1)$-connected $2n$-dimensional smooth manifold. We show that, for $n=4$, the homotopy inertia group of $M^{2n}$ is trivial and if $n=8$ and $H^n(M^{2n};\mathbb{Z})\cong \mathbb{Z}$, then the homotopy inertia group of $M^{2n}$ is also trivial. We further compute the group $\mathcal{C}(M^{2n})$ of concordance classes of smoothings of $M^{2n}$ for $n=8$. Finally, we show that if a smooth manifold $N$ is tangentially homotopy equivalent to $M^8$, then $N$ is diffeomorphic to the connected sum of $M^8$ and a homotopy $8$-sphere.
\end{abstract}
\section{Introduction}
Let $\Theta_m$ be the group of smooth homotopy spheres defined by M. Kervaire and J. Milnor in \cite{KM63}. Recall that the collection of homotopy spheres $\Sigma$  which admit a diffeomorphism $M\to M\#\Sigma$ form a subgroup $I(M)$ of $\Theta_m$, called the inertia group of $M$, where we regard the connected sum $M\#\Sigma^m$ as a smooth manifold with the same underlying topological space as $M$ and with smooth structure differing from that of $M$ only on an $m$-disc. The homotopy inertia group $I_h(M)$ of $M^m$ is a subset of the inertia group consisting of homotopy spheres $\Sigma$ for which the identity map $\rm{id}:M\to M\#\Sigma^m$ is homotopic to a diffeomorphism. Similarly, the concordance inertia group of $M^m$, $I_c(M^m)\subseteq \Theta_m$, consists of those homotopy spheres $\Sigma^m$ such that $M$ and $M\#\Sigma^m$ are concordant. Among these groups, only the concordance inertia group $I_c(M)$ is a homotopy type invariant of $M$ (\cite{Bru71, Kaw69}). Also, in \cite{Fra84}, it was shown that the homotopy inertia group $I_h(M)$ is an $h$-cobordism invariant for manifolds of dimension at least eight.\\
\indent The paper is organized as following. Preliminary definitions and notations are given in section 2. In section 3, we show that if $M$ and $N$ have the same homotopy type of simply connected closed smooth $m$-manifolds such that the integral and mod-$2$ cohomologies of $M$ vanish in odd degrees, then their homotopy inertia groups are equal. Thus, the homotopy inertia group $I_h(M)$ is a homotopy type invariant for closed oriented simply connected manifolds with vanishing the integral and mod-$2$ cohomology in odd degrees. In fact, we show that  if $M^{m}$  $(m\geq  5)$ is  a closed simply connected smooth $m$-manifold such that the integral and mod-$2$ cohomologies of $M^{m}$ vanish in odd degrees, then the natural forgetful map  $F_{Con}:\mathcal{C}^{Diff}(M)\longrightarrow \mathcal{S}^{Diff}(M)$ is injective, where $\mathcal{C}^{Diff}(M)$ is the group of concordance classes of smooth structures on $M^m$ and $\mathcal{S}^{Diff}(M)$ is the set of diffeomorphism classes of homotopy-smooth structures on $M^m$. Let $M^{2n}$ be a closed $(n-1)$-connected $2n$-dimensional smooth manifold. In this section, we also show that, for $n=4$, the homotopy inertia group $I_h(M^{2n})$ is trivial and if $n=8$ and $H^n(M^{2n};\mathbb{Z})\cong \mathbb{Z}$, then the homotopy inertia group $I_h(M^{2n})$ is also trivial.\\
 In section 4, we compute the $0^{th}$ reduced stable cohomotopy group $\widetilde{\pi}^0(M^{2n})$ for $n=4,8$ and the group $\mathcal{C}(M^{2n})$ of concordance classes of smoothings of $M^{2n}$ for $n=8$, which will be used in Section 5.\\
In section 5, by use of surgery theoretic structure sets, we give a diffeomorphism classification of closed smooth manifolds in the tangential homotopy type of $M^{2n}$. Recall that a homotopy equivalence of smooth manifolds of the same dimension $f:L\to N$ is called tangential if the vector bundles $f^{*}(T(N))$ and $T(L)$ are stably isomorphic. In \cite{MTW80} it was studied that under what conditions tangentially homotopy equivalent manifolds are necessarily homeomorphic. Examples of non-homeomorphic tangentially homotopy equivalent manifolds were given in \cite{Cro02, CN14}. In this section, we show that if $n\equiv 0~ {\rm mod~ 4}$, $n>0$ and a smooth manifold $N$ is tangential homotopy equivalent to $M^{2n}$, then $N$ is homeomorphic to $M^{2n}$. We also show that if a smooth manifold $N$ is tangentially homotopy equivalent to $M^{8}$, then $N$ is diffeomorphic to $M^8\#\Sigma$, where $\Sigma\in \Theta_8$ is a homotopy $8$-sphere.\\
\noindent
{\bf Notation: } We denote by $[X;Y]$ the set of all free homotopy classes of maps from $X$ to $Y$. If $X$, $Y$ are well pointed
spaces, and if $Y$ is $0$-connected, then the fundamental group $\pi_1(Y)$ acts on the set $[X;Y]_0$ of based homotopy classes; the set $[X;Y]$ of all free homotopy classes can be identified with the orbit set of this
action (\cite[Section III.1]{whi78}). If $Y$ is an $H$-space (or if $Y$ is 1-connected), this action is trivial, so $[X; Y]=[X;Y]_0$. Denote by $O=\colim_n O(n)$, $PL=\colim_n PL(n)$, $Top= \colim_n Top(n)$, $G=\colim_n G(n)$, the direct limit of the groups of orthogonal transformations, PL homeomorphisms, homeomorphisms and homotopy equivalences, respectively. We also denote $SPL=\colim_n SPL(n)$ and $SG=\colim_n SG(n)$, the direct limit of the groups of orientation-preserving PL homeomorphisms and homotopy equivalences, respectively.\\
In this paper, all the manifolds will be closed, smooth, oriented and connected, and all the homotopy equivalences, homeomorphisms and diffeomorphisms are assumed to preserve orientation, unless otherwise stated.

\section{Preliminaries}\label{sec2}
\begin{definition}\rm
		\begin{itemize}
			\item[(a)] A homotopy $m$-sphere $\Sigma^m$ is an oriented smooth closed manifold homotopy equivalent to the standard unit sphere $\mathbb{S}^m$ in $\mathbb{R}^{m+1}$.
			\item[(b)]A homotopy $m$-sphere $\Sigma^m$ is said to be exotic if it is not diffeomorphic to $\mathbb{S}^m$.
			\item[(c)] Two homotopy $m$-spheres $\Sigma^{m}_{1}$ and $\Sigma^{m}_{2}$ are said to be equivalent if there exists an orientation preserving diffeomorphism $f:\Sigma^{m}_{1}\to \Sigma^{m}_{2}$.
		\end{itemize}
		The set of equivalence classes of homotopy $m$-spheres is denoted by $\Theta_m$.  
\end{definition}
\begin{definition}\rm($Cat=Diff~~{\rm{or}}~~ PL$)
	Let $M$ be a closed $Cat$-manifold. Let $(N,f)$ be a pair consisting of a closed $Cat$-manifold $N$ together with a homeomorphism $f:N\to M$. Two such pairs $(N_{1},f_{1})$ and $(N_{2},f_{2})$ are concordant provided there exists a $Cat$-isomorphism $g:N_{1}\to N_{2}$ such that the composition $f_{2}\circ g$ is topologically concordant to $f_{1}$, i.e., there exists a homeomorphism $F: N_{1}\times [0,1]\to M\times [0,1]$ such that $F_{|N_{1}\times 0}=f_{1}$ and $F_{|N_{1}\times 1}=f_{2}\circ g$. The set of all such concordance classes is denoted by $\mathcal{C}^{Cat}(M)$.\\
	We will denote the class in $\mathcal{C}^{Cat}(M)$ of $(N,f)$ by $[N,f]$. The base point of $\mathcal{C}^{Cat}(M)$ is the equivalence class $[M,Id]$ of $Id : M\to M$.
\end{definition}
\begin{definition}\rm \label{smoo.stru}($Cat=Diff~~{\rm{or}}~~ Top$-structure sets)	Let $M$ be a closed $Cat$-manifold. We define the $Cat$-structure set $\mathcal{S}^{Cat}(M)$ to be the set of equivalence classes of pairs $(N,f)$ where $N$ is a closed $Cat$-manifold and $f: N\to M$ is a homotopy equivalence.
	And the equivalence relation is defined as follows:
	\begin{center}
		$(N_1,f_1)\sim  (N_2,f_2)$ if there is a $Cat$-isomorphism $h:N_1\to N_2$ \\
		such that $f_2\circ h$ is homotopic to $f_1$.
	\end{center}
	We will denote the class in $\mathcal{S}^{Cat}(M)$ of $(N,f)$ by $[N,f]$. The base point of $\mathcal{S}^{Cat}(M)$ is the equivalence class $[M,Id]$ of $Id : M\to M$. Denote by $[M^m\#\Sigma^m]$ the  equivalence class of $(M^m\#\Sigma^m, \rm{Id})$ in $\mathcal{S}^{Cat}(M)$ or $\mathcal{C}^{Diff}(M)$. (Note that $[M^m\#\mathbb{S}^m]$ is the class of $(M^m, \rm{Id})$.)
\end{definition}
\begin{definition}\rm{
		Let $M^m$ be a closed smooth $m$-dimensional manifold. The inertia group $I(M)\subset \Theta_{m}$ is defined as the set of $\Sigma \in \Theta_{m}$ for which there exists a diffeomorphism $\phi :M\to M\#\Sigma$.\\
		\indent Define the homotopy inertia group $I_h(M)$ to be the set of all $\Sigma\in I(M)$ such that there exists a diffeomorphism $M\to M\#\Sigma$ which is homotopic to $\rm{Id}:M \to M\#\Sigma$.\\
		\indent Define the concordance inertia group $I_c(M)$ to be the set of all $\Sigma\in I_h(M)$ such that $M\#\Sigma$ is concordant to $M$.}
\end{definition}
Observe that there are natural forgetful maps  $F_{Con} : \mathcal{C}^{Diff}(M)\to \mathcal{S}^{Diff}(M)$ and $F_{Diff} : \mathcal{S}^{Diff}(M)\to \mathcal{S}^{Top}(M)$. In \cite{Wei90}, Weinberger showed that the forgetful map $F_{Diff} : \mathcal{S}^{Diff}(M)\to \mathcal{S}^{Top}(M)$ is finite-to-one with image containing a subgroup of finite index. It also follows that the forgetful maps $F_{Diff} : \mathcal{S}^{Diff}(M)\to \mathcal{S}^{Top}(M)$ and $F_{Con} : \mathcal{C}^{Diff}(M)\to \mathcal{S}^{Diff}(M)$ fit into a short exact sequence of pointed sets:
$$\mathcal{C}^{Diff}(M)\stackrel{F_{Con}} {\longrightarrow} \mathcal{S}^{Diff}(M)\stackrel{F_{Diff}} {\longrightarrow} \mathcal{S}^{Top}(M).$$
\begin{definition}\rm (\cite{Bro72, Wal99, Luc02, Ran92, Ran02})
	Let $N$ and $M$ be closed smooth manifolds.
	A degree one normal map $(N,f,b):N\to M$ is a degree one map $f:N\to M$ together with a stable bundle map $b:\nu_N\to \xi$ where $\nu_N$ is the stable normal bundle of $N$, $b$ covers $f$ and $\xi$ is some stable vector bundle over $M$ (necessarily fibre homotopy equivalent to $\nu_M$).\\
	A normal bordism of degree one normal maps $(N_i,f_i,b_i)$, $i=0$,$1$ is a degree one normal map $(Z,g,c):Z\to M\times [0,1]$ restricting to $(N_i,f_i,b_i)$ over $M\times \{i\}$.\\
	The  set of normal bordism classes of degree one normal maps into a closed smooth manifold $M$ is called the normal structure set of $M$, and we denote it by $\mathcal{N}^{Diff}(M)$.
\end{definition}
The normal structure set $\mathcal{N}^{Diff}(M)$ of an $m$-dimensional smooth closed simply connected manifold $M^m$ fits into the
surgery exact sequence of pointed sets
\begin{equation}\label{tsurgry}
\cdots\longrightarrow\mathcal{N}^{Diff}(M\times [0,1])\stackrel{\sigma}{\longrightarrow}L_{m+1}(\mathbb{Z}) \stackrel{\theta}{\longrightarrow}\mathcal{S}^{Diff}(M) \stackrel{\eta}{\longrightarrow}\mathcal{N}^{Diff}(M)\stackrel{\sigma}{\longrightarrow} L_{m}(\mathbb{Z}),
\end{equation}
where the map $\eta:\mathcal{S}^{Diff}(M)\to \mathcal{N}^{Diff}(M)$ is defined by mapping $[N,f]$ to $[(N,f,b):N\to M]$, here $b:\nu_N\to (f^{-1})^{*}(\nu_N)$ is the canonical bundle map and $\xi=(f^{-1})^{*}(\nu_N)$, $L_{*}(\mathbb{Z})$ are the surgery obstruction groups, the maps $\sigma:\mathcal{N}^{Diff}(M)\to L_{m}(\mathbb{Z})$ and $\sigma:\mathcal{N}^{Diff}(M\times [0,1])\to L_{m+1}(\mathbb{Z})$ are the surgery obstruction maps and the map $\theta:L_{m+1}(\mathbb{Z})\to \mathcal{S}^{Diff}(M)$ is the action of $L_{m+1}(\mathbb{Z})$ on $\mathcal{S}^{Diff}(M)$. Using  identity maps as base points we have Sullivan's familiar identifications $\mathcal{N}^{Diff}(M)=[M,G/O]$ and $\mathcal{N}^{Diff}(M\times [0,1])=[SM,G/O]$
where $SM$ is the suspension of $M$. 
\begin{definition}\cite{MTW80,CH15}\rm
	Let $M$ be a closed smooth $m$-manifold with stable normal bundle $\nu_M$ of rank $k\gg m$. The smooth tangential structure set of
	$M$: $$\mathcal{S}^{TDiff}(M)=\{(N,f,b)~|~f:N\to M,~b:\nu_N\to \nu_M\}/\simeq,$$ 
	consists of equivalences classes of triples $(N,f,b)$, where $N$ is a smooth manifold, $f:N\to M$ is a homotopy equivalence and $b:\nu_N\to \nu_M$ is a map of stable bundles. Two structures $(N_0,f_0,b_0)$ and $(N_1,f_1,b_1)$
	are equivalent if there is an $s$-cobordism $(U, N_0, N_1, F, B)$ and where $F:U\to M$ is a simple homotopy equivalence, $B:\nu_U\to \nu_M$ is a bundle map and these data restrict to $(N_0,f_0,b_0)$ and $(N_1,f_1,b_1)$ at the boundary of
	$U$.
\end{definition}
The tangential surgery exact sequence for a simply connected closed smooth manifold $M$
finishes with the following four terms :
\begin{equation}\label{tsurgry}
L_{m+1}(\mathbb{Z}) \stackrel{\theta}{\longrightarrow}\mathcal{S}^{TDiff}(M) \stackrel{\eta^t}{\longrightarrow}\mathcal{N}^{TDiff}(M)\stackrel{\sigma}{\longrightarrow} L_{m}(\mathbb{Z}),
\end{equation}
where the definition of $\mathcal{N}^{TDiff}(M)$ is similar to the definition of $\mathcal{S}^{TDiff}(M)$ except
that for representatives $(N,f,b)$ we require only that $f:N\to M$ is a degree one map and the equivalence relation is defined using normal cobordisms over $(M,\nu_M)$. The set $\mathcal{N}^{TDiff}(M)$ of tangential normal invariants of $M$ can be identified with $[M, SG]$.\\
Let $QM=\Omega^{\infty}\Sigma^\infty M=\colim_n \Omega^{n}\Sigma^n M$. Then $QM$ is an infinite loop space; i.e., the zero space of the $\Omega$-spectrum $\{Q\Sigma^n M\}$. For $M=S^0$, $QS^0$ has components $Q_kS^0$, $k\in\mathbb{Z}$, determined by the degree of self maps of spheres. We let $Q_kS^0$ be the $k$th component of $QS^0$. The space $QS^0$ is an $H$-space under the loop product $$\ast:QS^0\times QS^0\mapsto QS^0$$ which satisfies $$\ast:Q_sS^0\times Q_tS^0\mapsto Q_{s+t}S^0,$$
and for any space $M$ there is a free and transitive action $$[M,Q_1S^0]\times [M,Q_0S^0]\stackrel{\ast}{\rightarrow} [M,Q_1S^0]~~([\phi],[\alpha])\mapsto [\phi]\ast [\alpha].$$ It is important to note that $Q_0S^0$ has the same homotopy type of $SG$ and the $H$-space structure on $SG$ given by composition of maps corresponds not to the loop product on $Q_0S^0$ (\cite{Tsu71}). Therefore, $[M, SG]=[M,Q_0S^0]=\widetilde{\pi}^0(M)$ as sets, where $\widetilde{\pi}^0(M)$ is the $0$th reduced stable cohomotopy group of $M$. Actually, $\widetilde{\pi}^0(M)$ is a ring, and, as groups, $[M, SG]\cong 1+\widetilde{\pi}^0(M)$ where the addition on the right is given by $(1+\alpha)(1+\beta)=1+\alpha+\beta+\alpha \beta$ (\cite{Har69}).\\

\section{The Homotopy Inertia Group}\label{sec3}
In \cite{Kas15}, the concordance inertia group $I_c(M^{2n})$ and the inertia group $I(M^{2n})$ of an $(n-1)$-connected $2n$-smooth manifold $M^{2n}$ were considered for $n=4, 8$. In this section, we compute the homotopy inertia group of $M^{2n}$.
\begin{theorem}\label{forget}
	Let $m\geq  5$ and $M^{m}$ be a closed simply connected smooth $m$-manifold such that the integral and mod-$2$ cohomologies of $M^{m}$ vanish in odd degrees. Then $F_{Con}:\mathcal{C}^{Diff}(M)\longrightarrow \mathcal{S}^{Diff}(M)$ is injective.
\end{theorem}
\begin{proof}
For $i\in \{1,2\}$, let $[N_i,f_i]\in \mathcal{C}^{Diff}(M)$, where $N_i$ is a
closed smooth manifold and $f_i$ is a homeomorphism
from $N_i$ to $M$ such that $F_{Con}([N_1,f_1])=F_{Con}([N_2,f_2])$. Thus
$(N_1,f_1)$ and $(N_2,f_2)$ represent the same element in $\mathcal{S}^{Diff}(M)$, and therefore, there exists a diffeomorphism $\phi$
from $N_1$ to $N_2$ such that $f_2\circ \phi\simeq f_1$.\\
\indent Since each $(N_i,f_i)$ represents an element in
$\mathcal{C}^{PL}(M)$ which is isomorphic to $H^3(M;\mathbb{Z}_2)=0=\left \{ [M,Id] \right \}$, there is a PL-homeomorphism $k_i$ from $N_i$ to $M$ such that $Id\circ k_i$ coincides up to topological concordance, and hence up to homotopy with $f_i$. Therefore, $k_2\circ \phi \simeq k_1$, say under the homotopy $H$. Define the homotopy
equivalence $\widetilde{H}:N_1\times I\to M\times I$ by $\widetilde{H}(x, t)=(H(x,t),t)$, where
$(x, t)\in N_1\times I$. By \cite[Theorem 1]{Sul67}, $\widetilde{H}$ is homotopic to a PL homeomorphism $({\rm mod}~N_1\times \partial I)$. It is because such a deformation is obstructed by cohomology classes in $H^{*}(N_1\times (I, \partial I);P_*)$. Since $N_1$ can have nonzero cohomology groups in even dimensions and $P_i$ is zero for odd $i$, these cohomology groups are all zero. Thus, $k_2\circ \phi$ is topological concordant to $k_1$ and hence $Id\circ k_2\circ \phi$ to $Id\circ k_1$, and $f_2\circ \phi$ to $f_1$. This completes the proof by noting that
both $(N_1, f_1)$ and $(N_2, f_2)$ represent the same element in $\mathcal{C}^{Diff}(M)$.
\end{proof}
As a corollary to the result above, we have the following.
\begin{corollary}\label{forget1}
	Let $m\geq  7$ and $M^{m}$ be a closed simply connected smooth $m$-manifold such that the integral and mod-$2$ cohomologies of $M^{m}$ vanish in odd degrees. Then $I_c(M)=I_h(M)$.
\end{corollary}
\begin{remark}\rm
	\indent
	\begin{itemize}
		\item[(1)]If $M^{2n}$ is a closed simply connected smooth $2n$-manifold such that the integral cohomology of $M^{2n}$ vanishes in odd degrees, by Corollary \ref{forget1}, $I_h(M^{2n})=I_c(M^{2n})$.
		\item[(2)] If $M^{2n}$ is a closed smooth manifold homotopy equivalent to $\mathbb{C}\textbf{P}^n$, then $I_c(M^{2n})=0$  for $n\leq 8$ \cite{Kaw68}. Therefore, by Corollary \ref{forget1}, $I_h(M^{2n})=0$ $n\leq 8$.
	\end{itemize}
\end{remark}
By \cite[Theorem 2.7((i),(iii))]{Kas15} and Corollary \ref{forget1}, we have the following.
\begin{corollary}
Let $M^{2n}$ be a closed smooth $(n-1)$-connected $2n$-manifold. 
\begin{itemize}
	\item[(1)] If $n=4$, then $I_h(M^{2n})=0$.
	\item[(2)] If $n=8$ and $H^n(M;\mathbb{Z})\cong \mathbb{Z}$, then $I_h(M^{2n})=0$.
	\end{itemize}	
\end{corollary}
Since the concordance inertia group is a homotopy invariant and by Corollary \ref{forget1}, we have the following result.
\begin{corollary}\label{homoto}
If $M$ and $N$ have the same homotopy type of simply connected closed smooth $m$-manifolds such that the integral and mod-$2$ cohomologies of $M$ vanish in odd degrees, then their homotopy inertia groups are equal.
\end{corollary}
\section{Smooth Structures on 7-connected 16-manifolds}\label{sec3}
In \cite[Theorem 2.7(ii)]{Kas15}, it was shown that
the group $\mathcal{C}(M^{2n})$ of concordance classes of smoothings of a closed smooth  $(n-1)$-connected $2n$-manifold $M^{2n}$ is isomorphic to the group of smooth homotopy spheres $\Theta_{2n}$ for $n=4$. In this section, we compute the group $\mathcal{C}(M^{2n})$ for $n=8$.\\
We denote $X^{2n}=\mathbb{S}^{n}\bigcup_{g}\mathbb{D}^{2n}$, where $g:\mathbb{S}^{2n-1}\to \mathbb{S}^{n}$ is an attaching map. Let $i_{n}:[\mathbb{S}^n;SPL]\to[\mathbb{S}^n;SG]$ and $\iota_{-}^*:[X, -] \to [\mathbb{S}^n, -]$ be the induced maps by the natural maps $i:SPL\to SG$ and $\iota:\mathbb{S}^n\to X$, respectively. Denote by $f_{X}:X\to \mathbb{S}^{2n}$, a degree one map.
\begin{theorem}\label{teclem2}
	\indent
\begin{itemize}
\item[(i)] $\iota_{PL}^{*}:[X^{16}, SPL] \to [\mathbb{S}^8, SPL]$ is surjective.
\item[(ii)] $[X^{16}, SG]$ is isomorphic to either $\mathbb{Z}_4\oplus \mathbb{Z}_2~~{\rm or}~~\mathbb{Z}_2\oplus \mathbb{Z}_2\oplus \mathbb{Z}_2$.
\end{itemize}
\end{theorem}
The proof of Theorem \ref{teclem2} requires one fact which we prove below:
\begin{lemma}\label{teclem1}
Let $n\equiv ~0~ {\rm mod~ 4}, n>0$. Then $i_n:[\mathbb{S}^n, SPL] \to [\mathbb{S}^n, SG]$ is an isomorphism.
\end{lemma}
\begin{proof}
	Consider the homotopy exact sequence for the fibration $SPL\to SG\to G/PL$:
	\begin{equation}\label{longSG}
	\cdots\longrightarrow \pi_{n+1}(SG)\stackrel{}{\longrightarrow}\pi_{n+1}(G/PL)\stackrel{}{\longrightarrow}\pi_{n}(SPL)\stackrel{i_n}{\longrightarrow}\pi_{n}(SG)\stackrel{}{\longrightarrow}\pi_{n}(G/PL)\to\cdots
	\end{equation}
	and by using the fact $\pi_n(SG)=\pi_n^s$ is a finite group, $\pi_{n+1}(G/PL)=0$ and $\pi_{n}(G/PL)=\mathbb{Z}$, we have that $i_n:[\mathbb{S}^n, SPL] \to [\mathbb{S}^n, SG]$ is an isomorphism.
\end{proof}
\paragraph{{\bf Proof of Theorem \ref{teclem2}:}}
(i): There is a commutative diagram
\begin{equation}\label{digram1}
 \begin{CD}
 @. @. @.  @. [\mathbb{S}^{16}, SPL]\\
 @. @.  @.   @.  @VV{\cong} V\\
 @. @. @.  @. [\mathbb{S}^{16}, SG]\\
 @. @.  @.   @.  @VVV\\
 @. @. @.  @. \mathbb{Z}\cong [\mathbb{S}^{16}, G/PL]\\
 @. @. @.   @.  @VVV\\
 [\mathbb{S}^{9},SPL]@>S(g)^{*}>> [\mathbb{S}^{16},SPL]@>f_{X}^{*}>> [X^{16}, SPL] @>\iota_{PL}^{*}>> [\mathbb{S}^{8}, SPL]  @>g^{*}>> [\mathbb{S}^{15}, SPL]\\
@VV i_{9}V @V{\cong}V i_{16}V  @VVV   @V{\cong}V i_{8}V     @VVi_{15}V\\
[\mathbb{S}^{9},SG]@>S(g)^{*}>> [\mathbb{S}^{16},SG]@>f_{X}^{*}>>[X^{16}, SG]@>\iota_{G}^{*}>> [\mathbb{S}^{8}, SG] @>g^{*}>>  [\mathbb{S}^{15}, SG]\\
@VVV @. @.  @.   @.\\
[\mathbb{S}^{9}, G/PL]=0 @. @.  @.   @.\\
 \end{CD}
\end{equation}
 in which, the homomorphism $$i_{n}:\pi_n(SPL)\to \pi_{n}(SG)$$ is an isomorphism by Lemma \ref{teclem1}, where $n=8, 16$ and the rows and columns are part of the long exact sequences obtained from the cofiber sequence $\mathbb{S}^{15}\to \mathbb{S}^{8}\to X$ and the fiber sequence $SPL\to SG\to G/PL$, respectively. From this commutative diagram, it is seen that  $$\iota_{PL}^{*}:[X^{16}, SPL] \to [\mathbb{S}^8, SPL]$$ is surjective if and only if $$g^*:\pi_8(SG)\to \pi_{15}(SG)$$  is the zero homomorphism. To show that $g^*:\pi_8(SG)\to \pi_{15}(SG)$  is the zero homomorphism, we will use Toda's notation \cite[p.189]{Tod62} for special elements in the stable stems $\pi_n^s$. Recall that $\pi_{*}^s$ equal the direct sum of the $\pi_n^s$ is an anti-commutative graded ring with respect to composition as multiplication and $\pi_n(SG)$ can be identified with the $n$th stable homotopy group of sphere $\pi_n^s$; i.e., $\pi_n(SG)=\pi_n^s$ for $n\geq 1$. This fact is used to identify the map $$g^*:\pi_8(SG)\to \pi_{15}(SG)$$
 with the map $$(S^q(g))^*:\pi_{8+q}(\mathbb{S}^q)\to \pi_{15+q}(\mathbb{S}^q),$$ where $S^k(g)$ is the $k$th iterated suspension of $g$. Since $\pi_8^s=\mathbb{Z}_{2}(\bar{v})\oplus\mathbb{Z}_{2}(\epsilon)$, where $\bar{v}$ and $\epsilon$ are the generators of the first and second components of $\pi_8^s$. Let $f:\mathbb{S}^{8+q}\to \mathbb{S}^{q}$  represents either $\bar{v}$ or $\epsilon$. First note that the homotopy class $$[S(g)]=aS(\sigma)+x \in \pi^s_7 \cong \mathbb{Z}_{16}\oplus\mathbb{Z}_{15},$$ where  $a\in\mathbb{Z}$, $S(\sigma)\in \mathbb{Z}_{16}$ is the generator represented by the suspension of the hopf map  $\sigma: \mathbb{S}^{15}\to\mathbb{S}^{8}$ and $x\in \mathbb{Z}_{15}$ has odd order. Using this together with the fact that $\pi_8^s=\mathbb{Z}_{2}(\bar{v})\oplus\mathbb{Z}_{2}(\epsilon)$ has order 4, we see that 
 \begin{align*}
 f\circ S^q(g)&=f\circ S^{q-1}(aS(\sigma)+x)\\
 &=a(f\circ S^{q-1}(S(\sigma)))+(f\circ S^{q-1}(x))\\
 &=a(f\circ S^{q-1}(S(\sigma)))=a(f\circ S^{q}(\sigma)).
 \end{align*}
  Since $S^{q}(\sigma)$ represents the element $S(\sigma)$ in $\pi_7^s$ and $\bar{v}\circ S(\sigma)=\epsilon\circ S(\sigma)=0$ in $\pi_{15}^s$ \cite[Theorem 14.1(iii), p.190]{Tod62}, we have $f\circ S^{q}(\sigma)=0$ in $\pi_{15}^s$. This implies that $f\circ S^q(g)=0$ in $\pi_{15}^{s}$. Therefore the map $(S^q(g))^*:\pi_{8+q}(\mathbb{S}^q)\to \pi_{15+q}(\mathbb{S}^q)$ is the zero homomorphism and hence $g^*:\pi_8(SG)\to \pi_{15}(SG)$ is the zero homomorphism.\\
 (ii):  Recall that the homotopy class $$[S(g)]=a\sigma+x \in \pi^s_7 \cong \mathbb{Z}_{16}\oplus\mathbb{Z}_{15},$$ where $a\in\mathbb{Z}$, $\sigma$ itself denotes the generator of $\mathbb{Z}_{16}$ represented by the suspension of the hopf map  $\sigma: \mathbb{S}^{15}\to\mathbb{S}^{8}$ and $x\in \mathbb{Z}_{15}$ has odd order. Using this together with the fact that $\pi_{16}=\mathbb{Z}_{2}(\eta^*)\oplus\mathbb{Z}_{2}(\eta\circ \rho)$ has order 4, we see that 
  ${\rm Im}(S(g)^{*}:[\mathbb{S}^9,SG]\to [\mathbb{S}^{16},SG])\subseteq \pi^s_{16}$ is generated by the following three elements: $$ \nu^3\circ a\sigma,~~~ \mu\circ a\sigma,~~~ \eta\circ \epsilon\circ a\sigma,$$ where $\nu^3$, $\mu$ and $\eta\circ \epsilon$ are generators of the first, second and third components of $$\pi_9^s=[\mathbb{S}^9, SG]\cong \mathbb{Z}_{2}\oplus\mathbb{Z}_{2}\oplus\mathbb{Z}_{2}$$ respectively.  Now, by \cite[Theorem 14.1((i),(ii)), p.190]{Tod62}, we have $$\nu \circ a\sigma=a(\nu \circ \sigma)=0,$$ $$\sigma\circ \epsilon=0$$ and $\sigma\circ \mu=\eta \circ \rho.$ Using this relations and anti-commutativity of $\pi^s_*$, we have the following relations: $$\nu^3\circ a\sigma=a(\nu^2\circ (\nu \circ \sigma))=0,$$ $$\eta\circ \epsilon\circ a\sigma=a(\eta\circ (\epsilon\circ \sigma))=a(\eta\circ (\sigma\circ \epsilon))=0,$$ 
 $$\mu\circ a\sigma=a(\mu\circ \sigma)=-a(\sigma \circ \mu)=-a(\eta \circ \rho)=a(\rho \circ \mu).$$ This implies that ${\rm Im}(S(g)^{*}:[\mathbb{S}^9,SG]\to [\mathbb{S}^{16},SG])$ is contained in the subgroup of $\pi^s_{16}$ generated by $\rho \circ \mu$. This shows that  ${\rm Im}(S(g)^{*}:[\mathbb{S}^9,SG]\to [\mathbb{S}^{16},SG])=\mathbb{Z}_2$. Now by the exact sequence in the horizontal rows of the commutative diagram (\ref{digram1}) along $SG$ and using the fact we have observed in the proof of (i) that $\iota_{G}^*:[X^{16}, SG]\to [\mathbb{S}^8, SG]$ is surjective, we get that
 $$[X^{16}, SG]\cong \mathbb{Z}_4\oplus \mathbb{Z}_2~~{\rm or}~~\mathbb{Z}_2\oplus \mathbb{Z}_2\oplus \mathbb{Z}_2.$$  This completes the proof of Theorem \ref{teclem2}.\\
 
 The following result follows from Theorem \ref{teclem2}.
\begin{corollary}\label{7con}
Let $M^{16}$ be a closed $7$-connected $16$-manifold such that $H^{8}(M;\mathbb{Z})\cong \mathbb{Z}$. Then $[M, SG]$ is isomorphic to either $\mathbb{Z}_4\oplus \mathbb{Z}_2~~{\rm or}~~\mathbb{Z}_2\oplus \mathbb{Z}_2\oplus \mathbb{Z}_2$.
\end{corollary}
\begin{theorem}\label{cohomo}
Let $M^{8}$ be a closed $3$-connected $8$-manifold. Then 
\begin{itemize}
\item[(i)] $[M, SG]\cong \mathbb{Z}_2\oplus \mathbb{Z}_2$.
\item[(ii)] $[M, G/O]\cong (\oplus_{i=1}^{k+1}~~\mathbb{Z})\oplus \mathbb{Z}_2$, where $k$ is the $4$th Betti number of $M$.
\end{itemize}
\end{theorem}
\begin{proof}
According to Wall \cite {Wal62}, $M$ has the homotopy type of $X=(\bigvee_{i=1}^{k} \mathbb{S}_{i}^4)\bigcup_{g}\mathbb{D}^{8}$, where $k$ is the $4$th Betti number of $M$, $\bigvee_{i=1}^{k} \mathbb{S}_{i}^4$ is the wedge sum of $4$-spheres and $g:\mathbb{S}^{7}\to \bigvee_{i=1}^{k} \mathbb{S}_{i}^4$ is the attaching map of $\mathbb{D}^{8}$. Let $f_X:X\to \mathbb{S}^{8}$ be the collapsing map obtained by identifying $\mathbb{S}^{8}$ with $X/\bigvee_{i=1}^{k} \mathbb{S}_{i}^4$. Then by the naturality of the Puppe sequence, we have the following exact ladder:\\
\begin{equation}\label{longSG}
\cdots\longrightarrow [\bigvee_{i=1}^{k}S \mathbb{S}_{i}^4, Cat]\stackrel{(S(g))^{*}}{\longrightarrow}[\mathbb{S}^{8}, Cat]\stackrel{f_{X}^{*}}{\longrightarrow}[X, Cat]\stackrel{\iota_{Cat}^{*}}{\longrightarrow}[\bigvee_{i=1}^{k} \mathbb{S}_{i}^4, Cat]\stackrel{g^{*}}{\longrightarrow}[\mathbb{S}^{7}, Cat],
\end{equation}
where $S(g)$ is the suspension of the map $g:\mathbb{S}^{7}\to \bigvee_{i=1}^{k} \mathbb{S}_{i}^4$.\\
Using the facts that $$[\bigvee_{i=1}^{k}S \mathbb{S}_{i}^4, Cat]\cong \prod_{i=1}^{k}[\mathbb{S}_{i}^{5}, Cat],$$ and $$[\bigvee_{i=1}^{k} \mathbb{S}_{i}^4, Cat]\cong \prod_{i=1}^{k}[\mathbb{S}_{i}^4, Cat],$$ 
the above exact sequence (\ref{longSG}) becomes 
\begin{equation}\label{longSG1}
\cdots \longrightarrow \prod_{i=1}^{k}[\mathbb{S}_{i}^{5}, Cat] \stackrel{(S(g))^{*}}{\longrightarrow}[\mathbb{S}^{8}, Cat] \stackrel{f_{X}^{*}}{\longrightarrow}[X, Cat]\stackrel{\iota_{Cat}^{*}}{\longrightarrow} \prod_{i=1}^{k}[\mathbb{S}_{i}^4, Cat].
\end{equation}
 If $Cat=SG$ in the exact sequence (\ref{longSG1}) and using the fact that $[\mathbb{S}^n, SG]=0,$ where $n=4,5$, we get that $$[\mathbb{S}^8, SG]\cong [X, SG]\cong \mathbb{Z}_2\oplus \mathbb{Z}_2.$$ Hence $$[M^8,SG]\cong \mathbb{Z}_2\oplus \mathbb{Z}_2.$$ This proves (i).\\
 We now take $Cat=G/O$ and by using the facts that $[\mathbb{S}^8, G/O]=\mathbb{Z}_2\oplus \mathbb{Z}$ and  $[\mathbb{S}^n, G/O]=0$ (\cite{Sul67}), where $n=5,7$, it follows that $$[X, G/O]\cong [M, G/O]\cong (\oplus_{i=1}^{k+1}~~\mathbb{Z})\oplus \mathbb{Z}_2.$$
This proves (ii). 
\end{proof}
Recall that $PL/O$ has a commutative $H$-space structure which makes the fibration $PL/O \to BSO \to BSPL$ into a sequence of $H$-space maps where $BSO$
and $BSPL$ have compatible commutative $H$-space structures coming from the Whitney
sum of bundles \cite{Lan00}. Associated to this fibration, we have the long exact Puppe sequences of abelian groups, for any space $X$,
\begin{equation}\label{digram2}
\begin{CD}
\cdots@>>>[X, SPL]@>\phi_{*}>> [X,PL/O]@>\partial_{X}>> [X, BSO] @>>> [X, BSPL],\\
\end{CD}
\end{equation}
where $\phi_{*}$ is induced by the natural map $\phi:SPL\to PL/O$. When $X=M$ is a smooth manifold, $\partial_{M}$ computes the 
difference a smooth structure makes to the isomorphism class of the stable tangent
bundle. That is, $$\partial_{M}([N,f])={f^{-1}}^{*}T^0(N)-T^0(M)\in \widetilde{KO}^0(M)=[M,BSO],$$ where $T^0(M)$ is the stable tangent bundle of $M$.\\
We now prove the following.
\begin{theorem}\label{smoothcayley}
	Let $M^{16}$ be a closed smooth $7$-connected  $16$-manifold such that $H^{8}(M;\mathbb{Z})\cong \mathbb{Z}$. Then $\mathcal{C}^{Diff}(M)$ contains exactly four elements.
\end{theorem}
\begin{proof}
	Since $PL/O$ fits into a fibration $$PL/O \to Top/O \to Top/PL\simeq K(\mathbb{Z}_2,3),$$ we have the following long exact sequence
	\begin{center} $\cdots[M,\Omega(Top/PL)]=[M,K(\mathbb{Z}_2,2)]=H^2(M;\mathbb{Z}_2)=0\to [M,PL/O]\to [M,Top/O]\to [M,Top/PL]=[M,K(\mathbb{Z}_2,3)]=H^3(M;\mathbb{Z}_2)=0.$
	\end{center}
	This implies that $[M,PL/O]\cong [M,Top/O]$. To show that $\mathcal{C}^{Diff}(M)$ contains exactly four elements, it is enough to show that $[M,PL/O]$ contains exactly four elements by using the identifications $\mathcal{C}^{Diff}(M^{16})=[M,Top/O]$ given by \cite[pp. 194-196]{KS77}. Since $M^{16}$ is a closed $7$-connected $16$-manifold such that $H^{8}(M;\mathbb{Z})\cong \mathbb{Z}$, then $M$ has the homotopy type of $X= \mathbb{S}^8\bigcup_{\alpha}\mathbb{D}^{16}$, where $\alpha:\mathbb{S}^{15}\to \mathbb{S}^8$ is the attaching map of $\mathbb{D}^{16}$.\\
	Consider the following commuting diagram of long exact sequences:
	\begin{equation}\label{digram3}
	\begin{CD}
	\cdots@>>>[X, SPL]@>\phi_{*}>> [X,PL/O]@>\partial_{X}>> [X, BSO] @>>> \cdots\\
	@. @VV\iota^*V  @VV\iota^*V  @VV\iota^*V   @.\\
	\cdots @>>>[\mathbb{S}^8, SPL]@>\phi_{*}>> [\mathbb{S}^8, PL/O]=\mathbb{Z}_2 @>\partial_{\mathbb{S}^8}>> [\mathbb{S}^8, BSO] @>>> \cdots\\
	\end{CD}
	\end{equation}
	and using the facts that $\phi_{*}:[\mathbb{S}^8, SPL]\to [\mathbb{S}^8, PL/O]$ and $\iota^*:[X, SPL]\to [\mathbb{S}^8, SPL]$ are surjective by result of \cite{Bru68} and Theorem \ref{teclem2} respectively, we have that  $\iota^*:[X, PL/O]\to [\mathbb{S}^8, PL/O]$ is  surjective. Since $[\mathbb{S}^8, PL/O]\cong \mathbb{Z}_2$ and the homomorphism $\iota^*:[X, PL/O]\to [\mathbb{S}^8, PL/0]$  fits into the following long exact Puppe sequences of abelian groups
	\begin{equation}\label{}
	\begin{CD}
	\cdots @>>>[\mathbb{S}^{16}, PL/O]@>f_{X}^{*}>> [X,PL/O]@>\iota^*>> [\mathbb{S}^{8}, PL/0] @>\sigma>> [\mathbb{S}^{15}, PL/O],\
	\end{CD}
	\end{equation}
	where $f_{X}^*:[\mathbb{S}^{16}, PL/O]=\Theta_{16}\to [X,PL/O]$ is monic by \cite[Theorem 2.7(iii)]{Kas15}. This implies that $[X,PL/O]\cong [M,PL/O]$ contains exactly four elements. This proves the theorem.
\end{proof}
In \cite{Kas15}, it was proved that there are, up to concordance, exactly two smooth structures on a closed smooth $3$-connected $8$-manifold. Now, Theorem \ref{smoothcayley} implies the following.
\begin{corollary}
Let $M^{16}$ be a closed smooth $7$-connected  $16$-manifold such that $H^{8}(M;\mathbb{Z})\cong \mathbb{Z}$. Then $M^{16}$ has exactly four distinct smooth structures up to concordance.	
\end{corollary}
\section{The Smooth Tangential Structure Set}\label{sec5}
In order to determine the smooth tangential structure set $\mathcal{S}^{TDiff}(M)$, one must compute the surgery obstruction map $\sigma$. We now prove the following.
\begin{theorem}\label{sur}
Let $M^{2n}$ be a closed smooth $(n-1)$-connected $2n$-manifold with $n\equiv 0~ {\rm mod~ 4}$, $n>0$.  Then the surgery obstruction $\sigma:[M^{2n}, SG]\to L_{2n}(\mathbb{Z})\cong \mathbb{Z}$ is the zero map.
\end{theorem}
 Let $\phi_*:[M, SG]\to[M, G/O]$ and $\psi_*:[M, Top/O]\to[M, G/O]$ be the induced maps by the natural maps $\phi:SG\to G/O$ and $\psi:Top/O\to G/O$ respectively. We prove the following.
\begin{theorem}\label{imge}
Let $M^{2n}$ be a closed smooth $(n-1)$-connected $2n$-manifold with $n\equiv 0~ {\rm mod~ 4}$, $n>0$. Then 
\begin{center}
	${\rm Im}(\phi_*:[M^{2n}, SG]\to[M^{2n}, G/O])={\rm Im}(\psi_*:[M^{2n}, Top/O]\to[M^{2n}, G/O])\cong [M^{2n}, Top/O].$
\end{center}
\end{theorem}
\begin{proof}
Consider the following commutative diagram :
\begin{equation}\label{digram3}
 \begin{CD}
@.   [M^{2n},\Omega(G/Top)]@. \\
  @.    @VVV    @. \\
  @.   [M^{2n},Top/O]@= [M^{2n},Top/O]\\
  @.  @VV\psi_{*}V   @VV\partial_{M}V \\
 [M^{2n},SG]@>\phi_{*}>> [M^{2n},G/O]@>\partial_{M}>> [M^{2n},BSO]\\
@| @VVV  @VVV  \\
   [M^{2n},SG]@>\phi_{*}>> [M^{2n},G/Top]@>\partial_{M}>> [M^{2n},BSTop]\\
 \end{CD}
 \end{equation}
 In this diagram, the rows and columns are the parts of the long exact sequences obtained from the fiber sequences $SG\to G/O\to BSO$, $SG\to G/Top\to BSTop$ and $Top/O\to G/O\to G/Top$.\\
Now the following three facts used in conjunction with a simple diagram chase in (\ref{digram3}) show that 
\begin{center}
${\rm Im}(\phi_*:[M^{2n}, SG]\to[M^{2n}, G/O])={\rm Im}(\psi_*:[M^{2n}, Top/O]\to[M^{2n}, G/O])\cong [M^{2n}, Top/O]$
\end{center}
 thus proving Theorem \ref{imge}.\\
\paragraph{{\bf Fact 1:}} The homomorphism $\partial_{M}:[M^{2n},G/Top]\to [M^{2n},BSTop]$ is monic for $n\equiv ~0~ {\rm mod~ 4}$.
\paragraph{{\bf Fact 2:}} The homomorphism $\psi_{*}:[M^{2n},Top/O]\to [M^{2n},G/O]$ is monic.
\paragraph{{\bf Fact 3:}}The homomorphism $\partial_{M}:[M^{2n},Top/O]\to [M^{2n},BSO]$ is the zero map.\\
It remains to verify these facts. Fact 1 is due to L. Kramer \cite[Proposition 9.9]{Kra03}.
\paragraph{{\bf Proof of Fact 2:}}
Consider the Puppe's exact sequence (\ref{longSG1}) for the inclusion $i:\bigvee_{i=1}^{i=k}\mathbb{S}_{i}^n\to M^{2n}$ along $Cat=\Omega(G/Top)$ and using the fact that $[\mathbb{S}^n, \Omega(G/Top)]=0$ for even $n$, we get that $[M^{2n},\Omega(G/Top)]=0$. Now it follows from the vertical exact sequence in (\ref{digram3}) that the homomorphism $\psi_{*}:[M^{2n},Top/O]\to [M^{2n},G/O]$ is monic. 
\paragraph{{\bf Proof of Fact 3:}}
Since $[M^{2n},Top/O]$ is a finite group and $\widetilde{KO}(M^{2n})=[M^{2n}, BSO]$ \cite{Yan11} is torsion-free, it follows that the homomorphism $\partial_{M}:[M^{2n},Top/O]\to [M^{2n},BSO]$ is the zero map.\\
\end{proof}
\paragraph{{\bf Proof of Theorem \ref{sur}:}}
The smooth structure set $\mathcal{S}^{Diff}(M^{2n})$ fits into the following surgery exact sequence 
\begin{equation}\label{surgry}
 \mathcal{S}^{Diff}(M) \xhookrightarrow{inj} [M, G/O]\stackrel{\sigma}{\longrightarrow} L_{2n}(\mathbb{Z}).
\end{equation}
Thus, a normal map $\varphi\in [M, G/O]$ represents a homotopy smoothing structure for $M$ if and only if $\sigma(\varphi)=0$. Observe that if $\varphi$ lies in the image $[M^{2n},Top/O]\xhookrightarrow{\psi_*}\mathcal{S}^{Diff}(M^{2n})\xhookrightarrow{inj} [M^{2n},G/O]$, then $\sigma(\varphi)=0$. Therefore the surgery obstruction $\sigma:{\rm Im}(\psi_*:[M^{2n}, Top/O]\to[M^{2n}, G/O])\to L_{2n}(\mathbb{Z})$ is the zero map. Combining this with the following commutative diagram (\cite{CH15}) and Theorem \ref{imge},
\begin{equation}\label{}
 \begin{CD}
 [M^{2n}, SG]@>\sigma>> L_{2n}(\mathbb{Z})\\
 @V\phi_*VV  @|  \\
   [M^{2n},G/O]@>\sigma>> L_{2n}(\mathbb{Z})\\
 \end{CD}
 \end{equation}
 we get that the surgery obstruction $\sigma:[M^{2n}, SG]\to L_{2n}(\mathbb{Z})\cong \mathbb{Z}$ is the zero map. This completes the proof of the theorem.
 \begin{theorem}\label{}
 Let $M^{2n}$ be a closed smooth $(n-1)$-connected $2n$-manifold with $n\equiv 0~ {\rm mod~ 4}$, $n>0$.  Then the map $\eta^t:\mathcal{S}^{TDiff}(M^{2n})\to [M^{2n},SG]$ is a set bijection.
 \end{theorem}
 \begin{proof}
Since $L_{2n+1}(\mathbb{Z})=0$, the tangential surgery exact sequence for $M^{2n}$ runs as follows:
\begin{equation}\label{tsurgry1}
0\longrightarrow \mathcal{S}^{TDiff}(M^{2n})\stackrel{\eta^t}{\longrightarrow}[M^{2n}, SG]\stackrel{\sigma}{\longrightarrow} L_{2n}(\mathbb{Z}).
\end{equation}
From Theorem \ref{sur}, it follows that the map $\eta^t:\mathcal{S}^{TDiff}(M^{2n})\to [M^{2n},SG]$ is a set bijection.
\end{proof}
\begin{theorem}\label{tanty}
Let $M^{2n}$ be a closed smooth $(n-1)$-connected $2n$-manifold with $n\equiv 0~ {\rm mod~ 4}$, $n>0$. Let $f:N\to M$ be a tangential homotopy equivalence where $N$ is a closed smooth manifold. Then there exist a smooth manifold $\widetilde{N}$ and a homeomorphism $g:\widetilde{N}\to M$  such that $(N, f)$ and $(\widetilde{N}, g)$ represent the same element in $\mathcal{S}^{Diff}(M)$.
 \end{theorem}
 \begin{proof}
 Observe that a tangential homotopy equivalence $f: N\to M$ represents an element of $\mathcal{S}^{Diff}(M)$ and its image in $\widetilde{KO}^0(M)=[M, BSO]$ is given by $(f^{-1})^{*}T^0(N)-T^0(M)=0$. Therefore, by using the third row exact sequence in the diagram (\ref{digram3}) and Theorem \ref{imge}, there exist a smooth manifold $\widetilde{N}$ and a homeomorphism $g:\widetilde{N}\to M$ such that $(N, f)$ and $(\widetilde{N}, g)$ represent the same map in $[M, G/O]$. Thus $(N, f)$ and $(\widetilde{N}, g)$ represent the same element in $\mathcal{S}^{Diff}(M)$.
\end{proof}
We note that Theorem \ref{tanty} implies the following.
\begin{corollary}\label{}
Let $M^{2n}$ be a closed smooth $(n-1)$-connected $2n$-manifold with $n\equiv 0~ {\rm mod~ 4}$, $n>0$. If a smooth manifold $N$ is tangential homotopy equivalent to $M$, then $N$ is homeomorphic to $M$.
\end{corollary}
\begin{theorem}\label{tan}
Let $M$ be a closed smooth $3$-connected $8$-manifold. If a smooth manifold $N$ is tangential homotopy equivalent to $M$, then there is a homotopy $8$-sphere $\Sigma$ such that $N$ is diffeomorphic to $M\#\Sigma$.
\end{theorem}
\begin{proof}
By \cite[Theorem 2.7(ii)]{Kas15}, we have $\mathcal{C}^{Diff}(M)=\left \{ [M\#\Sigma]~~ |~~\Sigma\in \Theta_{8} \right \}$. Now apply Theorem \ref{tanty}.
\end{proof}
As a consequence of \cite[Theorem 1.1]{Kas15} and Theorem \ref{tan}, we get
\begin{corollary}
Let $M$ be a closed smooth $3$-connected $8$-manifold such that $H^4(M;\mathbb{Z})\cong \mathbb{Z}$. If a smooth manifold $N$ is tangential homotopy equivalent to $M$, then $N$ is diffeomorphic to $M$.
\end{corollary}

\end{document}